\documentclass[12pt,a4paper]{article}
\usepackage[utf8]{inputenc}
\usepackage{amsmath}
\usepackage{amsthm}
\usepackage{amsfonts}
\usepackage{amssymb}
\usepackage{mathrsfs}

\usepackage{bbm}
\usepackage{graphicx}
\usepackage[all,cmtip]{xy}
\usepackage[bookmarks=true]{hyperref}
\usepackage[left=2.5cm,right=2.5cm,top=2cm,bottom=2cm]{geometry}

\theoremstyle{definition}
\newtheorem{theorem}{Theorem}[section]
\newtheorem{definition}[theorem]{Definition}
\newtheorem{proposition}[theorem]{Proposition}

\newtheorem{lemma}[theorem]{Lemma}

\newtheorem{remark}[theorem]{Remark}

\newtheorem{conjecture}[theorem]{Conjecture}

\usepackage{cancel} 
\usepackage{enumerate}
\usepackage{color}
\usepackage{xcolor}
\usepackage{verbatim}
\usepackage{amsmath}
\usepackage{amssymb}
\usepackage{pifont}
\usepackage{ulem}
\usepackage{enumerate}
\allowdisplaybreaks[4]

\def\address#1#2{\begingroup
	\noindent\parbox[t]{7.8cm}{%
		\small{\scshape\ignorespaces#1}\par\vskip1ex
		\noindent\small{\itshape E-mail address}%
		\/: #2\par\vskip4ex}\hfill%
\endgroup}%

\title{Uniform K-stability and Conformally K\"{a}hler, Einstein-Maxwell geometry on toric manifolds}
\author{Liu Yaxiong}

\begin{document}
     	
\maketitle

\begin{abstract}
Conformally K\"{a}hler, Einstein-Maxwell metrics and $f$-extremal metrics are generalization of canonical metrics in K\"{a}hler geometry.
We introduce uniform K-stability for toric K\"{a}hler manifolds,
and show that uniform K-stability is necessary condition for the existence of $f$-extremal metrics on toric manifolds.
Furthermore, we show that uniform K-stability is equivalent to properness of relative K-energy.
\end{abstract}

\tableofcontents

\section{Intrduction}
The subject of this paper is studying a special class of (non-K\"{a}hler in general) Hermitian metrics $\tilde{g}$ on a compact K\"{a}hler manifold $(M, J)$ of complex dimension $m\geq 2$.
Following \cite{AV-MG},
\begin{definition}
	A Hermitian metric $\tilde{g}$ on $(M,J)$ is called  \textit{conformally K\"{a}hler, Einstein-Maxwell} metric (cKEM metric for short) if there exist a smooth positive function $f$ on $M$ such that $g:=f^2\tilde{g}$ is a K\"{a}hler metric, and 
	\begin{flalign}
	&& \mathrm{Ric}^{\tilde{g}}(J\cdot, J\cdot) 
	& =\mathrm{Ric}^{\tilde{g}}(\cdot, \cdot), 
	&  \label{Killing condi} \\
	&& s_{\tilde{g}}  
	& = const,  
	&  \label{weit csc condi}
	\end{flalign}
	where $\mathrm{Ric}^{\tilde{g}}$ and $s_{\tilde{g}}$ denote the Ricci tensor and the scalar curvature of $\tilde{g}$.
\end{definition}
It is known that conformally K\"{a}hler, Hermitian metrics $\tilde{g}$ satisfying (\ref{Killing condi}) and(\ref{weit csc condi}) are in one-to-one correspondence with K\"{a}hler metrics $g$ which admit a Killing vector field $K$ with a positive Killing potential $f$ satisfying 
\begin{equation}
\label{weit csc}
s_{\tilde{g}}
=2 \dfrac{2 m-1}{m-1} f^{m+1} \Delta_{g}\left(\dfrac{1}{f}\right)^{m-1}+s_{g} f^{2}
= const ,
\end{equation} 
where $s_g$ is scalar curvature of $g$. 

Lebrun \cite{LC} proved that cKEM metric is a solution of Einstein-Maxwell equation in General Relativity when $\dim_{\mathbb{R}}=4$.
In \cite{AV-MG}, Apostolov-Maschler initiated a study of cKEM metric in a framework similar to the K\"{a}hler geometry, and set the existence problem of cKEM metrics in the Fujiki-Donaldson picture \cite{DS2}, \cite{FuA}.
In particular, they defined an obstruction generalizing Futaki invariant \cite{FA} in a fixed K\"{a}hler class for the existence of cKEM metric.
Futaki-Ono \cite{FO} and Lahdili \cite{LA18} extended independently the Licherowicz-Matsushima reductiveness theorem to the cKEM manifolds.
A number of recent existence results appear in \cite{FO1}, \cite{FO}, \cite{KT}, \cite{LA18}, \cite{deSou}. 

In \cite{FO}, Futaki-Ono give a definition of $f$-extremal K\"{a}hler metric as critical point of weighted Calabi functional, generalized the Calabi's extremal metric.
We can naturally extend the definitions of Futaki-Mabuchi bilinear form and weighted extremal vector field for $f$-extremal metric, see section \ref{Preliminaries}, also Lahdili \cite{LA18}.
~\\

In K\"{a}hler geometry, the existence of canonical metrics is conjectured that it is equivalent to a subtle stability condition of the underlying manifold in the sense of Mumford’s geometric invariant theory.
This is the so-called Yau-Tian-Donaldson conjecture, formulated as follows
\begin{conjecture}(Yau \cite{YS2}, Tian \cite{TG}, Donaldson \cite{DS})
	
	The polarized manifold $(M, L)$ should admit a cscK metric in the class $c_1(L)$ iff $(M, L)$ is K-polystable.
\end{conjecture}
Unfortunately examples in \cite{ACGT2} show that positivity of the Futaki invariant for algebraic test-configurations may not be enough to ensure the existence of a cscK metric. 
Sz\'{e}kelyhidi introduce a more strong concept of uniform stability. 
It becomes a candidate for the stability criterion of existence of cscK metrics.

In \cite{LA19}, Lahdili introduce the more general concept of constant weighted scalar curvature K\"{a}hler, and introduce weighted K-stability by giving an algebro-geometric definition of weighted Donaldson-Futaki invariant. 
We can consider the generalized Yau-Tian-Donaldson conjecture for weighted case.

The problem of searching canonical metrics may become simpler if the manifold admits more symmetry.
Hence, it is natural to consider the toric manifolds.
Each toric manifold $M^{2m}$ can be represented by a Delzant polytope $\triangle$ in $\mathbb{R}^m$ (see \cite{GV}), the equation for the extremal metrics becomes a real $4$th order equation on $\triangle$ via Abreu's formula of scalar curvature (see \cite{AM}), it is known as the Abreu equation.
In a series of papers, Donaldson initiated a program to study the cscK metric and extremal metrics on toric manifolds.
The problem is to solve the equation under certain necessary stability conditions of the pair $(\triangle, A)$ for some function $A$ on $\triangle$.
The formulation of stability becomes more elementary when we consider the toric manifolds.

Thus, in this paper, we focus on toric K\"{a}hler manifolds, consider the existence of $f$-extremal metrics.
As following the classical case (see \cite{Ch-Li-Sh}), we can define the (weighted) uniform K-stability.
Our main theorem is the easy direction of Yau-Tian-Donaldson type correspondence,

\begin{theorem}(see Theorem \ref{uniformly})
	If the weighted Abreu equation (\ref{weight Abreu eq}) has a solution in $\mathcal{S}$, then 
	($\triangle, \mathbf{L}, f$) is uniformly K-stable.
\end{theorem}
On the other hand, as a consequence of uniform K-stability, we have the properness of relative K-energy. In fact,  they are equivalent, which can be stated as follows:
\begin{proposition}(see Proposition \ref{unif equi proper})
	For any $\lambda>0$, the following are equivalent:
	\begin{enumerate}[(i)]
		\item 
		$\mathcal{F}_{(\triangle,f)}(u)\geq \lambda\|u\|_b$ for all $u\in \widetilde{\mathcal{C}}$, (i.e. ($\triangle, \mathbf{L}, f$) is uniform K-stable);
		\item 
		for all $0\leq \delta<\lambda$, there exists $C_{\delta}$ such that 
		$\mathcal{E}_{(\triangle,f)}(u)\geq \delta \|\pi(u)\|_b+C_{\delta}$ for all $u\in \mathcal{S}$, where $\pi(u)$ is the projection of $u$ to the space $\widetilde{\mathcal{S}}:=\widetilde{\mathcal{C}}\cap\mathcal{S}$,
	\end{enumerate}
    where $\mathcal{F}_{(\triangle,f)}(\cdot)$ is Donaldson-Futaki invariant and $\mathcal{E}_{(\triangle,f)}(\cdot)$ is weighted relative K-energy on the toric K\"{a}hler manifolds, we refer other notations to section \ref{The toric case}.
\end{proposition}
\subsection{Outline}

In section \ref{Preliminaries}, we introduce Futaki invariant from \cite{AV-MG} and definition of $f$-extremal metric from \cite{FO}.
we introduce Futaki-Mabuchi bilinear form and weighted extremal vector field (there are subtle difference from \cite{LA18}), and give a detail of proof.

In section \ref{The toric case}, we recall some elementary facts for toric K\"{a}hler manifolds, and introduce the definition of uniform K-stability for toric case, 
this fits within the general framework of toric cscK metrics developed by S. Donaldson \cite{DS} and B. Chen, A. Li and L. Sheng \cite{Ch-Li-Sh} 
(see (Definition \ref{def uniform} and Theorem \ref{uniformly})).

In section \ref{Proof of Theorem}, we provide a proof of Theorem \ref{uniformly}.

In section \ref{Uniform K-stability and properness of weighted relative K-energy}, we show that uniform K-stability is equivalent to properness of weighted relative K-energy on toric K\"{a}hler manifolds (see Proposition \ref{unif equi proper}).


\section{Preliminaries} 
\label{Preliminaries}
\subsection{Futaki invariant}
We consider a compact K\"{a}hler manifold $(M, J)$, a compact subgroup $G\subset \mathrm{Aut}_r(M, J)$ and fix a K\"{a}hler class  $\Omega$.
Denote by $\mathcal{K}^G_{\Omega}$ the space of $G$-invariant K\"{a}hler metrics $\omega$ in $\Omega$. 
General theory implies that for any $K\in \mathfrak{g}=\mathrm{Lie}(G)$, and any $\omega\in \mathcal{K}^G_{\Omega}$, we have
\[
\iota_K\omega=-df_{K,\omega}
\]
for a smooth function $f_{K,\omega}$, normalized by
\[
\int_M f_{K,\omega} v_{\omega}=a,
\]
where a is a fixed real constant, we denote by $f_{K, \omega, a}$ the unique function.

In what follows, we shall fix the vector field $K$ and, choosing a reference metric
$\omega\in \mathcal{K}^G_{\Omega}$,
a constant $a>0$ such that $f_{K,\omega,a}>0$ on $M$.

For any $\omega\in\mathcal{K}_{\Omega}^G$, denote 
$\tilde{g}_{\omega}(\cdot, \cdot)
:=(1/f^2_{K,\omega, a})\omega(\cdot, J\cdot)$, the corresponding scalar curvature of $\tilde{g}$ is 
\begin{equation}
\label{weighted scalar Kahler}
s_{\tilde{g}, \omega}
=2 \dfrac{2 m-1}{m-1} f^{m+1}_{K,\omega,a} \Delta_{\omega}\left(\dfrac{1}{f^{m-1}_{K,\omega,a}}\right)
+s_{\omega} f^{2}_{K,\omega,a},
\end{equation}
The K\"{a}hler metric $\omega\in\mathcal{K}^G_{\Omega}$ whose weighted scalar curvature $s_{\tilde{g}, \omega}$ is constant correspond to the cKEM metric with conformal factor $1/f^{2}_{K,\omega,a}$.
\begin{theorem}(see \cite{AV-MG})
	Let $(M,J)$ be a compact K\"{a}hler manifold of dimension $m\geq2$,
	$G\subset \mathrm{Aut}_r(M,J)$ a compact subgroup, $K\in\mathfrak{g}=\mathrm{Lie}(G)$ a vector field, and 
	$a>0$ a real constant such that $K$ has a positive Killing potential $f_{K,\omega,a}$ with integral equal to $a$
	with respect to any $\omega\in\mathcal{K}_{\Omega}^G$. Then, for any vector field $H\in\mathfrak{g}$ with Killing potential $h_{H,\omega,b}$ with respect to $\omega$, the integral
	\[
	\int_M s_{\tilde{g},\omega} h_{H,\omega,b} \left(\dfrac{1}{f_{K,\omega,a}} \right)^{2m+1}  v_{\omega}
	\]
	is independent of the choice of $\omega$. In particular,
	\[
	c_{\Omega,K,a}:=
	\dfrac{ \int_M s_{\tilde{g},\omega}  \left(\dfrac{1}{f_{K,\omega,a}} \right)^{2m+1}  v_{\omega} }{ \int_M   \left(\dfrac{1}{f_{K,\omega,a}} \right)^{2m+1}  v_{\omega} }
	\]
	is a constant independent of $\omega\in\mathcal{K}_{\Omega}^G$, and for any vector field $H\in\mathfrak{g}$,
	\begin{equation}
	\label{Futaki inv}
	\mathcal{F}^G_{\Omega,K,a}
	:= \int_M \dfrac{s_{\tilde{g},\omega}-c_{\Omega,K,a} }{f^{2m+1}_{K,\omega,a}} h_{H,\omega,b} v_{\omega},
	\end{equation}
	This is a linear functional called the \textit{Futaki invariant} associated to $(\Omega, G, K, a)$ on $\mathfrak{g}$ which is independent of the choice of $\omega\in\mathcal{K}_{\Omega}^G$ and $b$.
	Furthermore, $\mathcal{F}^G_{\Omega,K,a}$ must vanish if it exists a cKEM metric in $\Omega$.
\end{theorem}

\subsection{Weighted extremal vector field}
\label{Weighted extremal vector field}
Consider the Calabi functional
$\Phi: \mathcal{K}^G_{\Omega}\rightarrow\mathbb{R}$ 
defined by
\begin{equation}
\label{Calabi func 1}
\Phi(g)=\int_M s_{\tilde{g},\omega}^2 \dfrac{v_{\omega}}{f^{2m+1}_{K,\omega,a}}.
\end{equation}
If $g$ is a critical point of $\Phi$, the K\"{a}hler metric $g=\omega J$ is called an $f_{K,\omega,a}$-\textit{extremal} metric.
It follows that $g$ is an $f_{K,\omega,a}$-extremal metric iff $s_{\tilde{g},\omega}$ is Killing potential, due to Lemma 3.1 of \cite{FO}.

We introduce the Futaki-Mabuchi bilinear form and weighted extremal vector field for the $f$-extremal metric. 
Let $(M, g, J, \omega)$ be any (connected) compact K\"{a}hler manifold, fix a compact subgroup $G\subset\mathrm{Aut}_r(M, J)$ and fix a K\"{a}hler class $\Omega$.
Recall for any $K\in\mathfrak{g}=\mathrm{Lie}(G)$ and Killing potential $f_{K, \omega, a}$ with respect to $\omega\in\mathcal{K}^G_{\Omega}$ normalized by 
\[
\int_M f_{K, \omega, a} v_{\omega}=a,
\]
let $\omega_t=\omega_0+ dd^c\phi_t$ be any curve in $\mathcal{K}^G_{\Omega}$, 
the derivative  of $f_{K,\omega,a}$ along $\dot{\phi}$ in $T_g\mathcal{K}^G_{\Omega}$ is given by
\begin{equation}
\label{variation of Killing potential 1}
\dot{f}_{K,\omega,a}=-\mathcal{L}_{JK}\dot{\phi},
\end{equation}
this due to Lemma 4.5.1 of \cite{GP}.
\begin{remark}
	In the Lemma 4.5.1 of \cite{GP}, the function $f_{K,\omega}$ is normalized by
	\[
	\int_M f_{K,\omega} v_{\omega}=0,
	\]
	but this derivative formula is independent of the normalization condition.
\end{remark}
Hence, we have
\begin{equation}
\label{variation of Killing potential 2}
\dot{f}_{K,\omega,a}
=-d\dot{\phi}(JK)
=-g(JK, (d\dot{\phi})^{\sharp})
=g(df_{K,\omega_t,a}, d\dot{\phi}).
\end{equation}
\begin{theorem}
	\label{Futaki Mabuchi Binear form}
	Let $(M,J)$ be a compact K\"{a}hler manifold of dimension $m\geq2$,
	$G\subset \mathrm{Aut}_r(M,J)$ a compact subgroup, $K\in\mathfrak{g}=\mathrm{Lie}(G)$ a vector field, and 
	$a>0$ a real constant such that $K$ has a positive Killing potential $f_{K,\omega,a}$ with integral equal to $a$
	with respect to any $\omega\in\mathcal{K}_{\Omega}^G$. Then, for any vector field $H_i\in\mathfrak{g}$ with Killing potential $h_{H_i,\omega,0}$ with respect to $\omega$, $i=1, 2$,
	\begin{equation}
	\label{Futaki Mabuchi}
	\mathcal{B}^G_{\Omega,K,a}(H_1, H_2)
	=\int_M h_{H_1,\omega,0}h_{H_2,\omega,0} \dfrac{v_{\omega}}{f_{K,\omega,a}^{2m+1}}
	\end{equation}
	is independent of the choice of $\omega\in\mathcal{K}^G_{\Omega}$. 
	This bilinear form is called the \textit{Futaki-Mabuchi bilinear form} associated to $(\Omega, G, K, a)$ (c.f. \cite{Fu-Ma} and \cite{LA18}).
\end{theorem}
\begin{proof}
	Let $\omega_t=\omega_0+ dd^c\phi_t$ be any curve in $\mathcal{K}^G_{\Omega}$, for simplicity, denote $h_{i,t}:=h_{H_i, \omega_t, 0}$ and $f_t:=f_{K,\omega_t,a}$.
	The derivative  of $\mathcal{B}_t:=\mathcal{B}^G_{\Omega,K,a}(H_1, H_2)(g_t)$ 
	along $\dot{\phi}$ in $T_g\mathcal{K}^G_{\Omega}$ is given by
	\begin{align*}
	\dfrac{d}{dt}\mathcal{B}_t
	=& \dfrac{d}{dt} \int_M h_{1,t}h_{2,t} \dfrac{v_{\omega_t}}{f_t^{2m+1}}             \\
	=&\int_M \dot{h}_{1,t}h_{2,t} \dfrac{v_{\omega_t}}{f_t^{2m+1}}
	+\int_M h_{1,t}\dot{h}_{2,t} \dfrac{v_{\omega_t}}{f_t^{2m+1}}
	-\int_M h_{1,t}h_{2,t} \Delta_{g_t}\dot{\phi_t} \dfrac{v_{\omega_t}}{f_t^{2m+1}}            \\
	&-\int_M h_{1,t}h_{2,t} \dfrac{\dot{f_t}}{f_t^{2m+2}} v_{\omega_t}   \\
	=& \langle \dfrac{h_{2,t} dh_{1,t}}{f_t^{2m+1}}, d\dot{\phi_t} \rangle
	+\langle \dfrac{h_{1,t} dh_{2,t}}{f_t^{2m+1}}, d\dot{\phi_t}  \rangle
	-\langle \dfrac{h_{1,t}h_{2,t}}{f_t^{2m+1}}, \Delta_{g_t}\dot{\phi_t} \rangle 
	-\langle 
	\dfrac{h_{1,t} h_{2,t}}{f_t^{2m+2}}df_t, d\dot{\phi_t} 
	\rangle       \\
	=& \langle d\left(\dfrac{h_{1,t} h_{2,t}}{f_t^{2m+1}} \right), d\dot{\phi_t} \rangle
	-\langle \dfrac{h_{1,t}h_{2,t}}{f_t^{2m+1}}, \Delta_{g_t}\dot{\phi_t} \rangle     \\
	=&0,
	\end{align*}
	where $\langle\cdot, \cdot \rangle$ denotes the global inner product, and we have used (\ref{variation of Killing potential 2}) for the third equality.
\end{proof}
Notice that $\mathcal{B}^G_{\Omega,K,a}$ is positive definite on $\mathfrak{g}$ by our choice of $f_{K,\omega,a}$.

We restrict our attention to $\mathcal{K}^G_{\Omega}$ for some chosen maximal connected compact subgroup $G$ of $\mathrm{Aut}_r(M, J)$. 
We denote by $\mathfrak{g}\subset\mathfrak{h}_{red}$ the
Lie algebra of $G$: $\mathfrak{g}$ is then the space of hamiltonian Killing vector fields for all metrics $g$ in $\mathcal{K}^G_{\Omega}$, i.e. for any $X\in\mathfrak{g}$, 
$X=J\mathrm{grad}_gh_{X,\omega,0}
=\mathrm{grad}_{\omega}h_{X,\omega,0}$ for some smooth real function $h_{X,\omega,0}$ normalized by 
\[
\int_M h_{X,\omega,0} v_{\omega}=0.
\]
For any chosen metric $g$ in $\mathcal{K}^G_{\Omega}$,
we denote by $P^G_g$ the space of Killing potentials relative to $g$, i.e. the kernel of the fourth order Lichnerowicz operator $\mathbb{L}=\delta\delta D^{-}d$.

For simplicity, we denote $f_g=f_{K,\omega,a}$ as Theorem \ref{Futaki Mabuchi Binear form}.
We denote by $\Pi^G_{g,K,a}$ the orthogonal projecter with respect to the global inner product $\langle\cdot, \cdot \rangle_f$ of $C^{\infty}(M,\mathbb{R})$ onto $P^G_g$, where 
$\langle h_1,h_2 \rangle_f=\int_M h_1h_2\frac{v_{\omega}}{f_g^{2m+1}}$.

For any real function $h$, $\Pi^G_{g,K,a}(h)$ will be called \textit{the Killing part of} $h$ \textit{with
	respect to} $g$.
In particular, for any metric $g$ in $\mathcal{K}_{\Omega}$, $\Pi^G_{g,K,a}(s_{\tilde{g},\omega})$ is called the \textit{Killing part of the weighted scalar curvature of} $g$ and we thus get we get the following
$L^2$-orthogonal decomposition of $s_{\tilde{g},\omega}$:
\begin{equation}
\label{L2 decomposition}
s_{\tilde{g},\omega}=s_{\tilde{g},\omega}^G
+\Pi^G_{g,K,a}(s_{\tilde{g},\omega}),
\end{equation}
where $s_{\tilde{g},\omega}^G$ is $L^2$-orthogonal to $P^G_g$, call it the \textit{reduced weighted scalar curvature} with respect to $G$ and $K$.
It follows that $g$ is $f$-extremal if and only if the reduced weighted scalar curvature $s_{\tilde{g},\omega}^G$ is identically zero.

We now define a vector field $Z^G_{\Omega,K,a}$ in $\mathfrak{g}$ by
\begin{equation}
\label{weighted extremal v f}
\mathcal{F}^G_{\Omega,K,a}(X)
=\mathcal{B}^G_{\Omega,K,a}(X, Z^G_{\Omega,K,a}),
\end{equation}
for any $X$ in $\mathfrak{g}$.
Notice that $Z^G_{\Omega,K,a}$ is well-defined, as $\mathcal{B}^G_{\Omega,K,a}$ is positive definite on $\mathfrak{g}$.
By its very definition, $Z^G_{\Omega,K,a}$ only depends on the K\"{a}hler class $\Omega$ and of the choice of $G$ and $K$.
\begin{definition}
	\label{def of weighted extremal}
	Following the classical case,
	the above $Z^G_{\Omega,K,a}$ is called the \textit{weighted extremal vector field} of $\Omega$ relative to $G$ and $K$. 
	A function $\sigma$ is called the \textit{weighted extremal function} of $\Omega$ relative to $G$ and $K$ if $J\mathrm{grad}_g\sigma$ is weighted extremal vector field of $\Omega$ relative to $G$ and $K$.
\end{definition}
The following theorem extends the classical case of extremal vector field (c.f. Theorem 3.3.3 of \cite{Fubook} and \cite{LA18}).
\begin{theorem}
	\label{weighted extremal vector field}
	For any metric $g$ in $\mathcal{K}^G_{\Omega}$, the weighted extremal vector field $Z^G_{\Omega,K,a}$ is the gradient of the Killing part of the weighted scalar curvature:
	\begin{equation}
	\label{weighted extremal v f 1}
	Z^G_{\Omega,K,a}
	=J\mathrm{grad}_g(\Pi^G_{g,K,a}(s_{\tilde{g},\omega})).
	\end{equation}
	Moreover, the $L^2$-square norm of $\Pi^G_{g,K,a}(s_{\tilde{g},\omega})$ is given by
	\begin{equation}
	\label{L2 norm of Killing part}
	\int_M (\Pi^G_{g,K,a}(s_{\tilde{g},\omega}))^2 \dfrac{v_{\omega}}{f_g^{2m+1}}=E_{\Omega,K,a}^G,
	\end{equation}
	by setting
	\[
	E_{\Omega,K,a}^G
	=c_{\Omega,K,a}^2V_{\Omega,f} +\mathcal{F}^G_{\Omega,K,a}(Z^G_{\Omega,K,a}),
	\]
	where $V_{\Omega,f}=\int_M \frac{v_{\omega}}{f_g^{2m+1}}$.
	In particular, $\int_M (\Pi^G_{g,K,a}(s_{\tilde{g},\omega}))^2 \frac{v_{\omega}}{f_g^{2m+1}}$ is independent of $g$ in $\mathcal{K}^G_{\Omega}$
	and provides the following lower bound for the Calabi functional on $\mathcal{K}^G_{\Omega}$:
	\begin{equation}
	\label{lower bound of Calabi}
	\Phi(g)\geq E_{\Omega,K,a}^G,
	\end{equation}
	with equality if and only if $g$ is $f$-extremal.
\end{theorem}
\begin{proof}
	By its very definition, for any $g$ in $\mathcal{K}^G_{\Omega}$, 
	$Z^G_{\Omega,K,a}
	=J\mathrm{grad}_gh_{Z^G_{\Omega,K,a},\omega,0}$ for some real function $h_{Z^G_{\Omega,K,a},\omega,0}$, for simplicity, denoted by $h_0:
	=h_{Z^G_{\Omega,K,a},\omega,0}$.
	Then $h_0$ is determined by
	\[
	\mathcal{B}^G_{\Omega,K,a}
	(Z^G_{\Omega,K,a},X)
	=\int_M h_0 h_{X,\omega,0}
	\dfrac{v_{\omega}}{f_{g}^{2m+1}}
	=\mathcal{F}^G_{\Omega,K,a}(X)
	= \int_M (s_{\tilde{g},\omega}-c_{\Omega,K,a})  h_{X,\omega,0} \dfrac{v_{\omega}}{f^{2m+1}_g},
	\]
	for any $X$ in $\mathfrak{g}$ of Killing potential $h_{X,\omega,0}$. 
	Hence, we have
	\[
	\langle
	h_0-s_{\tilde{g},\omega}+c_{\Omega,K,a}, h_{X,\omega,0}
	\rangle_f=0.
	\]
	It follows that $h_0=\Pi^G_{g,K,a}(s_{\tilde{g},\omega})$ up to an additive constant.
	We thus get (\ref{weighted extremal v f 1}).
	
	We assume that $h_0=\Pi^G_{g,K,a}(s_{\tilde{g},\omega})-b$ for some constant $b$, i.e. 
	\[
	\int_M (\Pi^G_{g,K,a}(s_{\tilde{g},\omega})-b) \dfrac{v_{\omega}}{f_g^{2m+1}}=0.
	\] 
	It follows that
	\begin{align*}
	\mathcal{F}^G_{\Omega,K,a}(Z^G_{\Omega,K,a})
	=& \int_M (s_{\tilde{g},\omega}-c_{\Omega,K,a}) (\Pi^G_{g,K,a}(s_{\tilde{g},\omega})-b) \dfrac{v_{\omega}}{f^{2m+1}_{g}}            \\
	=&\int_M (\Pi^G_{g,K,a}(s_{\tilde{g},\omega}))^2 \dfrac{v_{\omega}}{f_g^{2m+1}}
	-b\int_M(s_{\tilde{g},\omega}-c_{\Omega,K,a}) \dfrac{v_{\omega}}{f_g^{2m+1}}              \\
	&-c_{\Omega,K,a}\int_M \Pi^G_{g,K,a}(s_{\tilde{g},\omega}) \dfrac{v_{\omega}}{f_g^{2m+1}}                 \\
	=&\int_M (\Pi^G_{g,K,a}(s_{\tilde{g},\omega}))^2 \dfrac{v_{\omega}}{f_g^{2m+1}}
	-c^2_{\Omega,K,a}V_{\Omega,f}.
	\end{align*}
	This gives (\ref{L2 norm of Killing part}).
	
	From (\ref{L2 decomposition}), we have
	\begin{align*}
	\Phi(g)=&\int_M s_{\tilde{g},\omega}^2 \dfrac{v_{\omega}}{f_g^{2m+1}}              \\
	=& \int_M (s^G_{\tilde{g},\omega})^2 \dfrac{v_{\omega}}{f_g^{2m+1}}  
	+\int_M |\Pi^G_{g,K,a}(s_{\tilde{g},\omega})|^2 \dfrac{v_{\omega}}{f_g^{2m+1}}             \\
	\geq& \int_M |\Pi^G_{g,K,a}(s_{\tilde{g},\omega})|^2 \dfrac{v_{\omega}}{f_g^{2m+1}}              \\
	=&E_{\Omega,K,a}^G,
	\end{align*}
	for any $g$ in $\mathcal{K}^G_{\Omega}$, with equality if and only if $s^G_{\tilde{g},\omega}\equiv0$,
	hence if and only if $g$ is $f$-extremal.
\end{proof}

\section{The toric case}
\label{The toric case}
\subsection{Weighted scalar curvature formula}
From now on, we consider that $M$ is toric K\"{a}hler manifold, i.e. $G=\mathbb{T}$ is an $m$-dimensional torus with effective action.
Let $\mathfrak{t}:=\mathrm{Lie}(\mathbb{T})$.
We fix a $\mathbb{T}$-invariant K\"{a}hler form $\omega$ in K\"{a}hler class $[\omega]$, a positive Killing potential $f$ for Killing vector field $K$, and vary the K\"{a}hler metrics within 
\[
\mathcal{C}^{\mathbb{T}}_{\omega}
:=\{ \mathbb{T}\text{-invariant, }\omega\text{-compatible complex structures on }(M, \omega, \mathbb{T}) \}.
\]
For any two elements 
$J, J^{\prime}\in \mathcal{C}^{\mathbb{T}}_{\omega}$, they are biholomorphic under a $\mathbb{T}$-equivariant 
diffeomorphism $\Phi$ which acts trivially on the cohomology class of $[\omega]$, i.e. $\Phi^{*}\omega$ is a
K\"{a}hler form in $\mathcal{K}^{\mathbb{T}}_{[\omega]}$ on $(M, J)$,
where
\[
\mathcal{K}^{\mathbb{T}}_{[\omega]}
:=\{\mathbb{T}\text{-invariant K\"{a}hler form in }[\omega] \}.
\]
For any two elements $\omega, \omega^{\prime}\in \mathcal{K}^{\mathbb{T}}_{[\omega]}$, there exist a $\mathbb{T}$-equivariant differmorphism $\Phi$ such that $\Phi^{*}\omega=\omega^{\prime}$ by the equivariant Moser lemma.

As the action of $\mathbb{T}\subset \mathrm{Aut}_r(M)$, then for each $K\in\mathfrak{t}$, $K$ admits a hamiltonian function with respect to $\omega$, it is therefore hamiltonian action.
So one can
use the Delzant description \cite{DT} of ($M, \omega, \mathbb{T}$) in terms of the corresponding momentum
image $\triangle=\mu(M) \subset \mathfrak{t}^{*}$
and a set of non-negative defining affine functions
$\mathbf{L}=(L_1, \cdots, L_d)$
for $\triangle$, defined on the vector space
$\mathfrak{t}^{*}\cong \mathbb{R}^m$. 

Let $\mu$ be the moment map of $\mathbb{T}$ action on $(M, \omega)$. $\triangle:=\mu(M)$, denote $\triangle^0$ the interior of $\triangle$.
We denote by $M^0:=\mu^{-1}(\triangle^0)$ the open dense subset of $M$, and $\mathbb{T}$ acts freely on $M^0$. 
The complex structure $J$ in $\mathcal{C}^{\mathbb{T}}_{\omega}$ and the corresponding K\"{a}hler metrics $g_J$ on $M^0$ can be described by using so-called momentum-angle coordinates, due to V. Guillemin \cite{GV}.
In this description, We fix a basis $\{e_1, \cdots, e_m \}$ of $\mathfrak{t}$ and denote by
$K_j = X_{e_j}$ the induced fundamental vector fields, denote $\{e_1^{*}, \cdots, e_m^{*} \}$ the its dual basis and write 
$x=(x_1, \cdots, x_m)$ for the elements of $\mathfrak{t}^{*}$.
For each point $p\in M^0$, we shall also identify the coordinate function $x_i=\langle x, e_i \rangle$ with the momentum
function $\langle \mu, e_i \rangle$ on $M$.
We know that the action of 
$\mathbb{T}\cong (\mathbb{S}^1)^m$ on $(M, J)$ extends
to an (effective) holomorphic action of the complex torus 
$\mathbb{T}^{\mathbb{C}}\cong (\mathbb{C}^{*})^m$.
Fixing a point $p_0\in M^0$, we can identify $M^0$ with the orbit 
$\mathbb{T}^{\mathbb{C}}(p_0)\cong (\mathbb{C}^{*})^m$. 
Using the polar coordinates $(r_i, t^0_i)$ on each $\mathbb{C}^{*}$, this identification gives rise the so-called \textit{angular coordinates}
\[
\textbf{t}=(t_{1}^{0}, \ldots, t_{m}^{0}): 
M^{0} \rightarrow \mathbb{T}.
\]
The functions 
$\{x_1, \cdots, x_m; t_1, \cdots, t_m \}$ 
on $\triangle^0\times\mathbb{T}$
are called \textit{momentum-angle coordinates} associated to $(g, J)$.
The symplectic 2-form $\omega$ then becomes
\begin{equation}
\label{symp form}
\omega=\sum_{i=1}^{m} dx_{i} \wedge dt_i, 
\end{equation}
whereas the K\"{a}hler metric is
\begin{equation}
\label{riem metric}
g_J=\sum_{i, j=1}^{m}
\left(G_{i j} dx_{i} \otimes dx_{j}
+H_{i j} dt_i \otimes dt_j \right) , 
\end{equation}
the complex structure is 
\begin{equation}
\label{cpx struc} 
J dt_i
=-\sum_{j=1}^{m} G_{ij}(x) dx_j, 
\end{equation}
where $\mathbf{G}=(G_{ij})=\mathrm{Hess}(u)$ for some smooth, strictly convex function $u(x)$ on $\triangle^0$, \textit{called \textbf{symplectic potential}} of $g_J$, and $\mathbf{H}=(H_{ij})=\mathbf{G}^{-1}$.

From the point view of this description, the symplectic potential of the induced canonical toric
K\"{a}hler structure $(g_0, J_0)$ via the Delzant construction is
\begin{equation}
\label{canonical potential}
u_{0}(x)=\dfrac{1}{2}\sum_{j=1}^d L_{j} \log L_{j}.
\end{equation}

In order to extend the K\"{a}hler structure on $M^0$ to $M$, we have the necessary and sufficient condition due to \cite{ACGT}, as follows
\begin{theorem}
	\label{iff cond}
	Let $(M, \omega)$ be a compact toric symplectic $2m$-manifold with momentum map 
	$\mu: M\rightarrow \triangle\subset \mathfrak{t}^{*}$, 
	and $\mathbf{H}$ be a positive definite $S^2\mathfrak{t}^{*}$-valued function on $\triangle^0$. 
	Then $\mathbf{H}$ defines a $\mathbb{T}$-invariant, $\omega$-compatible almost K\"{a}hler metric $g$ on $M$ via
	(\ref{riem metric}) iff it satisfies the following conditions:
	\begin{enumerate}[(i)]
		\item [(i)]
		[smoothness] $\mathbf{H}$ is the restriction to $\triangle^0$ of a smooth $S^2\mathfrak{t}^{*}$-valued function on $\triangle$;
		\item [(ii)]
		[boundary values] for any point $\xi$ on the facet $F_j\subset\triangle$ with inward normal $u_j=d L_j\in\mathfrak{t}$,
		\[
		\mathbf{H}_{\xi}(u_{j}, \cdot)=0 \quad 
		\text { and } \quad
		(d \mathbf{H})_{\xi}(u_{j}, u_{j})=2 u_{j},
		\]
		where the differential $d\mathbf{H}$ is viewed as a smooth 
		$S^2\mathfrak{t}^{*}\otimes \mathfrak{t}$-valued function on $\triangle$;
		\item [(iii)]
		[positivity] for any point $\xi$ in the interior of a face $F\subset \triangle$, $\mathbf{H}_{\xi}(\cdot,\cdot)$ is positive definite when viewed as a smooth function with values in $S^2(\mathfrak{t}/\mathfrak{t}_F)^{*}$. 
		Where we denote $\mathfrak{t}_F\subset\mathfrak{t}$ (for any face $F\subset\triangle$) the vector subspace
		spanned by the inward normals 
		$u_j\in \mathfrak{t}$ to facets containing $F$. Thus the tangent plane to points in the interior $F^0$ of $F$ is the annihilator $\mathfrak{t}_{F}^{0} \cong 
		(\mathfrak{t} / \mathfrak{t}_{F})^{*}$ of $\mathfrak{t}_F$ in $\mathfrak{t}^{*}$.
	\end{enumerate}
\end{theorem}
It follow that, up to a $\mathbb{T}$-equivariant isometry, the space $\mathcal{C}^{\mathbb{T}}_{\omega}$  
can be identified with the function space $\mathcal{S}(\triangle, \mathbf{L})$ 
($\mathcal{S}$ for simplicity)
of all smooth, strictly convex functions $u$ on $\triangle^0$ such that 
$\mathbf{H}^u:=\mathrm{Hess}(u)^{-1}=(u^{ij})$ 
satisfies the conditions of Theorem \ref{iff cond}.

A key feature of this description is
the remarkably simple expression for the scalar curvature of K\"{a}hler metrics, found by Abreu \cite{AM}. 
The scalar curvature $s_J$ of $g_J$ is given by
\begin{equation}
\label{simple scalar}
s_J=-\sum_{i, j=1}^{m} \mathbf{H}_{ij, ij}^{u} 
=-\sum_{i, j=1}^{m}\dfrac{\partial^2 G^{ij}}{\partial x_i\partial x_j},
\end{equation}
where we denote by $\varphi_{,i}$ the partial derivative $\partial\varphi/\partial x_i$ for a smooth function $\varphi$ of $x=(x_1, \cdots, x_m)$, and  $(G^{ij})=(G_{ij})^{-1}$, $G_{ij}=u_{ij}$.
From (\ref{cpx struc}), one has
\begin{flalign}
&&d d_{J}^{c} \varphi
&=d(Jd\varphi)
=d(\varphi_{,j} Jdx_j)
=\sum_{i, j, k=1}^{m}\left(\varphi_{,j} \mathbf{H}_{j k}^{u}\right)_{,i} dx_{i} \wedge d t_{k}     & \nonumber                               \\
&&\Delta_{J} \varphi
&=-\Lambda(dd^c_J\varphi)
=-\sum_{i=1}^m\iota_{(\frac{\partial}{\partial t_i})}
\iota_{(\frac{\partial}{\partial x_i})}dd^c\varphi
& \nonumber                                    \\
&& &=-\sum_{i, j=1}^{m}\left(\varphi_{,i} 
\mathbf{H}_{i j}^{u}\right)_{, j} 
=-\sum_{i, j=1}^{m}\varphi_{, ij}\mathbf{H}_{ij}^{u}
-\sum_{i, j=1}^{m} \varphi_{,i} \mathbf{H}_{i j, j}^{u}
& \label{laplace}
\end{flalign}

From now on, $f$ will be fixed affine function which is positive on $\triangle$ so that pull-back of $f$ by moment map $\mu$ is a Killing potential for any K\"{a}hler metric $g_J$ corresponding to an element $J\in \mathcal{C}^{\mathbb{T}}_{\omega}$. 
In particular, $f_{,ij}=0$ (see Lemma \ref{Killing affine}).  
Thus, for any smooth matrix-valued function
$\mathbf{H}=(\mathbf{H}_{ij})$ on $\triangle$ one computes
\begin{align*}
\sum_{i, j=1}^{m}\left( \dfrac{\mathbf{H}_{ij}}{f^{2 m-1}}\right)_{,i j}
=&\dfrac{1}{f^{2m-1}} \sum_{i, j=1}^{m} \mathbf{H}_{i j, i j}
+\dfrac{2m(2 m-1)}{f^{2m+1}} \sum_{i, j=1}^{m} \mathbf{H}_{i j} f_{, i} f_{, j}                 \\
&-\dfrac{2(2m-1)}{f^{2m+1}} \sum_{i, j=1}^{m} \mathbf{H}_{ij,j} f_{, i}   .
\end{align*}
From (\ref{weit csc}), (\ref{simple scalar}) and (\ref{laplace}) and the above equality, we gets a explicity expression for the scalar curvature $s_{J,f}$ of $(1/f^2)g_J$
\begin{equation}
\label{weight scalar}
\dfrac{s_{J,f}}{f^{2m+1}}
=-\sum^m_{i,j=1}\left(\dfrac{1}{f^{2m-1}}\mathbf{H}^u_{ij} \right)_{,ij}.
\end{equation} 

\subsection{Uniform K-stability on toric manifolds}
Let ($M$, $J$, $\omega$, $\mathbb{T}$) be a toric K\"{a}hler manifold, let ($\triangle, \textbf{L}$) be its momentum polytope, 
the standard Lebesgue measure $d\mu=dx_1\wedge\cdots\wedge dx_m$ on $\mathfrak{t}^{*}\cong\mathbb{R}^m$ and the affine
labels $\mathbf{L}=(L_1, \cdots, L_d)$ of $\triangle$ induce a measure $d\sigma$ on each facet $F_i\subset \partial\triangle$ by letting
\begin{equation}
\label{Lebesgue mea}
dL_i\wedge d\sigma=-d\mu.
\end{equation} 



For any $X\in G=\mathbb{T}$ and $\omega\in\mathcal{K}^G_{\Omega}$, the Killing potential $f_{X,\omega,a}$ is $\mathbb{T}$-invariant by $\mathbb{T}$-equivariant Moser Lemma.
\begin{lemma}
	\label{Killing affine}
	Suppose $(g, J)$ is a $\mathbb{T}$-invariant, $\omega$-compatible K\"{a}hler metric on $(M, \omega, \mathbb{T})$, corresponding to a symplectic potential $u\in\mathcal{S}(\triangle, \mathbf{L})$. 
	Then a smooth function $f$ is Killing potential with respect to $g$ if and only if $f=\mu^{*}\varphi$ for some affine function $\varphi$ on $\triangle$.
\end{lemma}
\begin{proof}
	As $f$ is $\mathbb{T}$-invariant, then it is the pull back by the moment map of a smooth function $\varphi(x)$ on $\triangle$.
	It thus follows from (\ref{symp form}) that on $M$,
	$\mathrm{grad}_{\omega}f=\sum_i\varphi_{,i}K_i$.
	Since each $K_i$ preserves $J$, the condition $\mathcal{L}_{\mathrm{grad}_{\omega}f}J=0$ reads as
	\begin{align*}
	0=JK_j\cdot \varphi_{,i}=(d\varphi_{,i})(JK_j)
	=&-\sum_{i,k}\varphi_{, ik}Jdx_k(K_j)
	=-\sum_{i,k,l}\varphi_{, ik}H_{kl}\theta_l(K_j)  \\
	=&-\sum_{i,k}\varphi_{, ik}H_{kj}.
	\end{align*}
	Since $\mathbf{H}$ is non-degenerate on $\triangle^0$, it follows that $\varphi_{, ik}=0$, i.e. $\varphi(x)$ must be an affine-linear function on $\triangle^0$, and hence on $\triangle$.
	
	Conversely, for any affine linear function 
	$\varphi(x)=\langle \xi,x \rangle+\lambda$, $\mathrm{grad}_{\omega}\mu^{*}\varphi=\sum_i \xi_iK_i$ which preserves $J$.
\end{proof}

\begin{lemma}(\cite{AV-MG})
	\label{integ by part}
	Let $\textbf{H}$ be any smooth $S^2\mathfrak{t}^{*}$-valued function on $\triangle$ which satisfies the boundary
	conditions of Theorem \ref{iff cond} , but not necessarily the positivity condition. Then, for any
	smooth functions $\varphi, \psi$ on $\mathfrak{t}^{*}$,
	\[
	\int_{\triangle}\left(\sum_{i, j=1}^{m} (\psi \mathbf{H}_{ij})_{,i j}\right) \varphi\ d \mu
	=\int_{\triangle}\left(\sum_{i, j=1}^{m} 
	(\psi \mathbf{H}_{i j}) \varphi_{,i j}\right) d \mu
	-2 \int_{\partial \triangle} \varphi \psi d \sigma.
	\]
\end{lemma}
Apply the above lemma for $\phi=1/f^{2m-1}$ and $\mathbf{H}=\mathbf{H}^u$ for some $u\in\mathcal{S}(\triangle, \mathbf{L})$, one has 
\begin{flalign}
&&-\int_{\triangle}\left(\sum_{i, j=1}^{m} (\dfrac{1}{f^{2m-1}} \mathbf{H}_{ij})_{,i j}\right) \varphi\ d \mu
=&-\int_{\triangle}\dfrac{1}{f^{2m-1}}\left(\sum_{i, j=1}^{m} 
\mathbf{H}_{i j}\varphi_{,i j}\right) d \mu
& \nonumber                 \\
&& 
&+2 \int_{\partial \triangle} \dfrac{\varphi}{f^{2m-1}} d \sigma. 
& \label{integral by part for f}
\end{flalign}
By the definition of weighted extremal function, let $\phi=h_0+b$ be a weighted extremal function as in the proof of Thm \ref{weighted extremal vector field}. 
$\phi$ is unique by normalization.
By Lemma \ref{Killing affine}, $\phi=\mu^{*}s_{(\triangle, \mathbf{L}, f)}$ for some affine linear function $s_{(\triangle, \mathbf{L}, f)}$ on $\triangle$. 
We call it the \textit{weighted extremal affine function} of ($\triangle, \textbf{L}$).
For any affine function $\varphi$ on $\triangle$, we have
\begin{align*}
&2\int_{\partial\triangle}\varphi \dfrac{d\sigma}{f^{2m-1}}
-\int_{\triangle}s_{(\triangle, \textbf{L}, f)} \varphi \dfrac{d\mu}{f^{2m+1}}         \\
=&\int_{\triangle}\dfrac{1}{f^{2m-1}}\left(\sum_{i, j=1}^{m} 
\mathbf{H}_{i j}\varphi_{,i j}\right) d \mu
-\int_{\triangle}\left(\sum_{i, j=1}^{m} (\dfrac{1}{f^{2m-1}} \mathbf{H}_{ij})_{,i j} 
+\dfrac{s_{(\triangle, \mathbf{L}, f)}}{f^{2m+1}} \right) \varphi\ d \mu              \\
=&\int_{\triangle} (s_{u,f}-s_{(\triangle, \mathbf{L}, f)}) \varphi \dfrac{d\mu}{f^{2m+1}}  \\
=&\dfrac{1}{(2\pi)^m} \int_M
(s_{\tilde{g},\omega}-\phi)\mu^{*}\varphi \dfrac{v_{\omega}}{f^{2m+1}}                \\
=&0.
\end{align*}
Thus, we can extend the definition of weighted extremal affine function to any compact convex simple labelled polytope as follows:
\begin{definition}
	\label{def weighted affine func}
	Suppose ($\triangle, \textbf{L}$) is a compact convex simple labelled polytope in $\mathfrak{t}^{*}$. 
	Affine-linear function
	$s_{(\triangle, \textbf{L}, f)}$ on $\mathfrak{t}^{*}$ is called the \textit{weighted extremal affine function} of 
	($\triangle, \textbf{L}$) 
	if for any affine linear function $\varphi$, such that
	\begin{equation}
	\label{extr aff cond}
	2\int_{\partial\triangle}\varphi \dfrac{d\sigma}{f^{2m-1}}
	-\int_{\triangle}s_{(\triangle, \textbf{L}, f)} \varphi \dfrac{d\mu}{f^{2m+1}}=0.
	\end{equation}
\end{definition}
\begin{proposition}
	\label{extr aff func}
	Suppose ($\triangle, \textbf{L}$) is a compact convex simple labelled polytope in $\mathfrak{t}^{*}$. 
	Then, there exists a unique weighted extremal affine function $s_{(\triangle, \mathbf{L}, f)}$ of ($\triangle, \textbf{L}$).
	Furthermore, if for $u\in\mathcal{S}$, the corresponding metric is $f$-\textit{extremal}, i.e. satisfies
	\[
	-f^{2m+1}\sum^m_{i,j=1}\left(\dfrac{1}{f^{2m-1}}\textbf{H}^u_{ij} \right)_{,ij}
	=s_{u,f}(x)=\langle \xi,x\rangle+\lambda,
	\]
	then the affine-function $s_{u,f}(x)$ must be equal to $s_{(\triangle, \textbf{L}, f)}$.
\end{proposition}
\begin{proof}
	Writing
	\[
	s_{(\triangle, \mathbf{L}, f)}
	=a_{0}+\sum_{j=1}^{m} a_{j} x_{j},
	\]
	the condition (\ref{extr aff cond}) gives rise to a linear system with positive-definite symmetric matrix since the coefficient matrix is just the Gram matrix for the $L^2$-inner product restricted to the subspace spanned by the functions 
	$\{ \frac{1}{f^{m+1/2}}, \frac{x_1}{f^{m+1/2}}, \cdots, \frac{x_m}{f^{m+1/2}} \}$
	\begin{align*}
	a_{0} \int_{\triangle} x_{i}\dfrac{d\mu}{f^{2m+1}}
	+\sum_{j=1}^{m} a_j \int_{\triangle} x_{j} x_{i}  \dfrac{d\mu}{f^{2m+1}}
	=&2 \int_{\partial\triangle}x_{i} \dfrac{d\sigma}{f^{2m-1}},             \\
	a_{0} \int_{\triangle} \dfrac{d\mu}{f^{2m+1}}
	+\sum_{j=1}^{m} a_j \int_{\triangle} x_{j}   \dfrac{d\mu}{f^{2m+1}}
	=&2 \int_{\partial\triangle} \dfrac{d\sigma}{f^{2m-1}},
	\end{align*}
	which therefore determines $(a_0, \cdots, a_m)$ uniquely.
	
	By Lemma \ref{integ by part}, for affine linear function $\varphi$, $s_{u,f}(x)$ satisfies  the defining property (\ref{extr aff cond}). By the uniqueness of weighted extremal affine function, thus the affine-function $s_{u,f}(x)$ must be equal to $s_{(\triangle, \textbf{L}, f)}$.
\end{proof}

We introduce the non-linear PDE
\begin{equation}
\label{weight Abreu eq}
s_{u,f}:=
-f^{2m+1}\sum^m_{i,j=1}\left(\dfrac{1}{f^{2m-1}}\textbf{H}^u_{ij} \right)_{,ij}
=s_{(\triangle, \textbf{L}, f)},
\end{equation}
this equation is called  the \textit{weighted Abreu equation}. If ($\triangle, \textbf{L}$) is a Delzant polytope, solutions of (\ref{weight Abreu eq})
correspond to $f$-extremal $\mathbb{T}$-invariant, $\omega$-compatible K\"{a}hler metrics on the toric
symplectic manifold $(M, \omega, \mathbb{T})$.

Thus, the existence problem of cKEM metrics can be reduced to finding symplectic potential $u$, such that
$u$ satisfies equation (\ref{weight Abreu eq}) for
$s_{(\triangle, \textbf{L}, f)}=c$,
where $c$ is a constant determined by ($\triangle, \textbf{L}$).
For this case, Apostolov and Maschler proved that existence of this equation implies K-stable in \cite{AV-MG}.
In the following, we will prove a generalized version for existence of solutions of weighted Abreu equation and uniform K-stability.
\begin{definition}
	We define the functional
	\[
	\mathcal{F}_{(\triangle,f)}(u)
	:=2\int_{\partial\triangle} u \dfrac{d\sigma}{f^{2m-1}} 
	-\int_{\triangle}s_{(\triangle, \textbf{L}, f)} u \dfrac{d\mu}{f^{2m+1}}
	\]
	acting on the space of continuous function on $\triangle$. $\mathcal{F}_{(\triangle,f)}$ is called the \textit{Donaldson-Futaki invariant}. 
	We know $\mathcal{F}_{(\triangle,f)}$ vanishes on affine linear function from Proposition \ref{extr aff func}.
\end{definition}

We introduce several classes of convex function
$\mathcal{C}(\triangle)$,
$\mathcal{PL}$,  $\mathcal{C}_{\infty}$ and $\mathcal{C}_{*}$ on $\triangle$. 
Let $\mathcal{C}(\triangle)$ denote the set of continuous convex functions on $\triangle$,
$\mathcal{C}_{\infty}\subset\mathcal{C}(\triangle)$ the subset of continuous convex functions on $\triangle$ which are smooth in the interior. 
Denoted by $\mathcal{PL}$ the set of all piecewise linear convex function on $\triangle$. 
Let $\mathcal{G}$ denote the set of affine linear functions, which acts on $\mathcal{C}(\triangle)$ and $\mathcal{C}_{\infty}$ by translation.
We essentially consider $\mathcal{C}_{\infty}/\mathcal{G}$, introduce a subset $\widetilde{\mathcal{C}}\subset\mathcal{C}_{\infty}$ that is isomorphic to $\mathcal{C}_{\infty}/\mathcal{G}$.
This can be done as follows:
fix a point $x_0\in \triangle^0$ and consider the normalized functions (in Donaldson \cite{DS}):
\[
\widetilde{\mathcal{C}}
=\{u\in \mathcal{C}_{\infty} | u\geq u(x_0)=0 \}.
\] 
Then any $u$ in $\mathcal{C}_{\infty}$ can be written uniquely as $u=\pi(u)+g$, where $g$ is affine linear and $\pi(u)\in\widetilde{C}$ for a linear projection $\pi$.
For any $u\in \widetilde{\mathcal{C}}$ we set
\[
\|u\|_b :=\int_{\partial\triangle} u \dfrac{d\sigma}{f^{2m-1}}.
\]
We can check easily that $\|\cdot\|_b$ is a norm on $\widetilde{\mathcal{C}}$, where $b$ stands for boundary. 
For all $u\in \widetilde{\mathcal{C}}$, we have
\begin{equation}
\label{bdd by L1}
\|u\|_b\geq C_0\int_{\partial\triangle} u d\sigma
\geq C\int_{\triangle}u d\mu
=C\|u\|_{L^1},
\end{equation}
on the other hand, 
\[
\|\cdot\|_b\leq C\|\cdot\|_{\infty}.
\]
\begin{definition}
	We say that norm $\|\cdot\|$ on $\widetilde{\mathcal{C}}$ is \textit{tamed} if there exists $C > 0$ such that on $\widetilde{\mathcal{C}}$ 
	\[
	\dfrac{1}{C}\|\cdot\|_{1} \leq \|\cdot\| \leq C\|\cdot\|_{\infty},
	\]
	where $\|\cdot\|_1:=\int_{\triangle}\cdot dx$ is the $L^1$-norm and $\|\cdot\|_{\infty}$ is the $C^0$-norm on $\widetilde{\mathcal{C}}$.
\end{definition}
\begin{remark}
	(i) The norm $\|\cdot\|_b$ is tamed.
	
	(ii) For any tamed norm $\|\cdot\|$, both the spaces of PL convex
	functions and smooth convex functions on the whole of $\triangle$ are dense in
	$\mathcal{C}^{*}(\triangle):=
	\{f\in\mathcal{C}(\triangle): f(x)\geq f(x_0)=0 \}$
\end{remark}

Let $\triangle^{*}$ be the union of $\triangle^0$ and its codimension-one faces. Let $\mathcal{C}_1$ be the set of positive convex functions $f$ on $\triangle^{*}$ such that 
\[
\int_{\partial\triangle} f d\sigma <\infty,
\] 
hence, we can extend $\|\cdot\|_b$ to space $\mathcal{C}_1$.
\begin{proposition}(Proposition 5.2.6 in \cite{DS})
	Suppose that $f_n$ is a sequence of function in $\widetilde{\mathcal{C}}$ with
	\[
	\int_{\partial\triangle} f_n d\sigma \leq C.
	\]
	Then there is a subsequence which converges, uniformly over compact subset of $\triangle^0$, to a convex function $f$ which has a continuous extension to a function $f^{*}$ on $\triangle^{*}$, defining an element of $\mathcal{C}_1$ with 
	\[
	\int_{\partial\triangle} f^{*} d\sigma 
	\leq \lim\inf \int_{\partial\triangle} f_n d\sigma.
	\]
\end{proposition}
So we can identify $f$ and $f^{*}$ and view $f$ as a function in $\mathcal{C}_1$. 
We define
\begin{equation}
\begin{split}
\mathcal{C}_{*}
=\{u\in\mathcal{C}_1\ |\ \exists \text{ constant } C>0 \text{ and a sequence of }u^{(k)} \text{ in }\widetilde{\mathcal{C}} 
\text{ s.t. } \ \|u^{(k)}\|_b<C 
\\ \text{ and locally uniformly converges to }u \text{ in }\triangle^0
\}.
\end{split}
\end{equation}
Let $P>0$ be a constant, we define
\[
\mathcal{C}_{*}^{P}
=\left\{u \in \mathcal{C}_{*}\ |\ \| u\| _ { b } \leq P\right\}.
\]
The functional $\mathcal{F}_{(\triangle,f)}$ can be generalized to the classes $\mathcal{C}_{\infty}$, $\mathcal{C}_{*}$ and $\mathcal{C}_{*}^P$.

\begin{lemma}(\cite{Ch-Li-Sh})
	For any $u\in \mathcal{C}_{*}^P$, there is a sequence of functions 
	$u_k\in \mathcal{C}_{\infty}$ 
	such that $u_k$ locally uniformly converges to $u$ in $\triangle^0$ and
	\begin{equation}
	\label{norm converg}
	\|u\|_{b}
	=\lim _{k \rightarrow \infty} \left\|u_k\right\|_{b},
	\end{equation}
	\begin{equation}
	\label{Futaki converg}
	\mathcal{F}_{(\triangle,f)}(u)
	=\lim_{k \rightarrow \infty} \mathcal{F}_{(\triangle,f)}(u_k).
	\end{equation}
\end{lemma}
\begin{proof}
	Without loss of generality, we can assume that $0$ is the center of the mass of $\triangle$ and $u$ is normalized at the origin, i.e. $u\geq u(0)=0$.
	
	Let $\tilde{u}_k(x)=u(r_k x)$, with $r_k<1$ and $\lim_{k\rightarrow\infty}r_k=1$.
	Then the sequence $\{\tilde{u}_k \}$ in $\mathcal{C}(\triangle)$ 
	locally uniformly converges to $u$. From the construction and the convexity we know $u(x)\geq\tilde{u}_k(x)$ for any $x\in\triangle$.
	Thus $\frac{1}{f^{2m-1}}u(x)\geq\frac{1}{f^{2m-1}}\tilde{u}_k(x)$,
	then the Lebesgue dominated convergence theorem implies that $\frac{1}{f^{2m-1}}\tilde{u}_k(x)$ converges to $(\frac{1}{f^{2m-1}}u)|_{\partial\triangle}$ in $L^1(\partial\triangle)$ as $k\rightarrow\infty$. 
	In particular,
	\begin{equation}
	\label{L1 converge}
	\lim_{k \rightarrow \infty} \int_{\partial\triangle} |\tilde{u}_k-u| \dfrac{d\sigma}{f^{2m-1}} =0.
	\end{equation}
	
	Let $\tilde{u}_k^{\varepsilon}$ be the  convolution $\tilde{u}_k\star \rho_{\varepsilon}(x)$, 
	where $\rho_{\varepsilon}(x)\geq 0$ is a smooth mollifier of $\mathbb{R}^m$ whose support is in $B_{\varepsilon}(0)$. 
	For any $k$ and 
	\[
	\varepsilon< \dfrac{1-r_k}{4} dist(0, \partial\triangle),
	\]
	$\tilde{u}_k^{\varepsilon}$ is a smooth convex function in $\triangle^0$. 
	In fact, for any $x, x_1\in\triangle^0$, and $t\in(0, 1)$
	\begin{align*}
	\tilde{u}_k^{\varepsilon}(tx+(1-t)x_1)
	=&\int_{\mathbb{R}^m} \tilde{u}_k(tx+(1-t)x_1-y) \rho_{\varepsilon}(y) dy                          \\
	=&\int_{B_{\varepsilon}(0)} \tilde{u}_k(t(x-y)+(1-t)(x_1-y)) \rho_{\varepsilon}(y) dy     \\
	\leq& \int_{B_{\varepsilon}(0)} \big[ t\tilde{u}_k(x-y)+ (1-t)\tilde{u}_k(x_1-y)\big] \rho_{\varepsilon}(y) dy      \\
	=&t\tilde{u}_k^{\varepsilon}(x) +(1-t)\tilde{u}_k^{\varepsilon}(x_1).
	\end{align*}
	Since $\tilde{u}_k$ is continuous on $\Omega_k$, we know that $\tilde{u}_k^{\varepsilon}$ uniformly converges to $\tilde{u}_k$ on compact subset of $\Omega_k$, where $\Omega_k$ is domain of $\tilde{u}_k$. In fact, $\triangle$ is compact subset of $\Omega_k$.
	Hence, then $\tilde{u}_k^{\varepsilon}$ uniformly converges to $\tilde{u}_k$ in $\triangle$ as $\varepsilon\rightarrow 0$. 
	We choose $\varepsilon_k>0$ such that
	\begin{equation}
	\label{L infty}
	\|\tilde{u}_k^{\varepsilon_k}  -\tilde{u}_k\|_{L^{\infty}(\triangle)} \leq \dfrac{1}{k}
	\end{equation}
	Let $u_k=\tilde{u}_k^{\varepsilon_k}$. Since
	\[
	\int_{\partial\triangle} |u-u_k| \dfrac{d\sigma}{f^{2m-1}}
	\leq \int_{\partial\triangle} |u-\tilde{u}_k| \dfrac{d\sigma}{f^{2m-1}}
	+\int_{\partial\triangle} |\tilde{u}_k-u_k| \dfrac{d\sigma}{f^{2m-1}},
	\]
	by (\ref{L1 converge}) and (\ref{L infty}), we conclude that $u_k$ converges to $u|_{\partial\triangle}$ in $L^1(\partial\triangle)$. 
	We proved (\ref{norm converg}).
	On the other hand, $u_k$ locally uniformly converges to $u$ in $\triangle^0$.
	Then (\ref{Futaki converg}) still holds.
	The lemma is proved.
\end{proof}

\begin{definition}
	\label{def uniform}
	($\triangle, \mathbf{L}, f$) is called \textit{uniformly K-stable} if there exists a constant $\lambda> 0$ such that
	\[
	\mathcal{F}_{(\triangle,f)}(u)\geq 
	\lambda \|u\|_b
	=\lambda \int_{\partial\triangle}u \dfrac{d\sigma}{f^{2m-1}},
	\quad \forall u\in \widetilde{\mathcal{C}}.
	\]
\end{definition}

\begin{theorem}
	\label{uniformly}
	If the weighted Abreu equation (\ref{weight Abreu eq}) has a solution in $\mathcal{S}$, then 
	($\triangle, \mathbf{L}, f$) is uniformly K-stable.
\end{theorem}

\section{Proof of Theorem \ref{uniformly}}
\label{Proof of Theorem}
Assume that $v\in \mathcal{S}$ is the solution of the weighted Abreu equation (\ref{weight Abreu eq}).

Recall that for any convex function $u$ on a domain $\Omega \subset \mathbb{R}^n$, a Monge–Amp\`{e}re measure $M_u$ on $\Omega$ is defined.

Let $u$ be a convex function. For any segment $I\Subset \triangle$, the convex function $u$ defines a convex function $w:=u|_I$ on $I$. It defines a Monge–Amp\`{e}re measure $M_w$ on $I$,we denote this by $N$.

\begin{lemma}
	\label{lower bound}
	Let $u\in \mathcal{C}^K_{*}$ and $u^{(k)}\in \mathcal{C}_{\infty}$ locally uniformly converges to $u$. If $N(I)=m>0$, then
	\[
	\mathcal{F}_{(\triangle,f)}(u^{(k)})> \tau m
	\]
	for some positive constant $\tau$ independent of $k$.
\end{lemma}	
\begin{proof}
	First assume that $u\in \mathcal{C}_{\infty}$.Then
	\[
	\mathcal{F}_{(\triangle,f)}(u)
	=2\int_{\partial \triangle} u \dfrac{d\sigma}{f^{2m-1}}   
	-\int_{\triangle}s_{(\triangle, \textbf{L}, f)} u \dfrac{d\mu}{f^{2m+1}} 
	=2\int_{\partial \triangle} \dfrac{u}{f^{2m-1}} d \sigma
	+\int_{\triangle} \left( 
	\dfrac{1}{f^{2m-1}}v^{ij} 
	\right)_{,ij}u d \mu.
	\]
	By Lemma \ref{integ by part}, take 
	$\psi=\frac{1}{f^{2m-1}}$, 
	we have 
	\begin{equation}
	\label{Fut if exist solu}
	\mathcal{F}_{(\triangle,f)}(u)
	=\int_{\triangle} \dfrac{v^{i j}}{f^{2m-1}} 
	u_{i j} d \mu.
	\end{equation}
	
	Let $p$ be the midpoint of $I$. For simplicity, we assume that $p$ is the origin, $I$ is on the $x_1$ axis and $I=(-a,a)$. Suppose that there is a Euclidean ball $B:=B_{\epsilon_0}(0)$ in $x_1=0$ plane such that $I\times B \Subset \triangle$. Consider the functions
	\[
	w_{x}(x_{1})=u(x_{1}, x), \quad x_{1} \in I,\ x \in B.
	\]
	For any $x\in B$, $w_x$ are convex functions on $I$. We denote the Monge–Ampere measure on $I$ induced by $w_x$ by $N_x$. Note that $N_0=N$. By the weak convergence of Monge–Ampere measure, we know that there exists a small $B$ such that for any $x\in B$
	\begin{equation}
	N_x(I)\geq m/2.
	\end{equation}
	
	On the other hand, the eigenvalues of $\frac{v^{ij}}{f^{2m-1}}$ are bounded below in $I\times B$ since $v$ is smooth. Let $\delta$ be the lower bound. Then
	\begin{flalign}
	&& \mathcal{F}_{(\triangle,f)}(u) 
	& \geq \int_{I\times B} \dfrac{v^{i j}}{f^{2m-1}} 
	u_{i j} d \mu
	\geq \delta\int_{I\times B} \mathrm{Tr}(u_{ij}) d\mu &\ \\
	&&  &\geq \delta\int_{I\times B} u_{11} d\mu 
	=\delta\int_B N_x(I) dx    
	\geq \dfrac{m\delta}{2} Vol(B). &\
	\end{flalign}
	This completes the proof for $u\in \mathcal{C}_{\infty}$.
	
	Now suppose that $u$ is a limit of a sequence $u^{(k)}\in\mathcal{C}_{\infty}$. Then $u^{(k)}$ converges to $u$ uniformly on $I\times B$. For each $u^{(k)}$ we repeat the above argument: let $w^{(k)}_x, N^{(k)}_x$ replace $w_x$ and $N_x$. Since $w^{(k)}_x$ converges to $w_x$ uniformly with respect to $k$ and $x$, $N^{(k)}_x(I)$ also converges to $N_x(I)$ uniformly with respect to $k$ and $x$. Hence, we conclude that when $k>K$ for some large $K$ and $x\in B$,
	\[
	N^{(k)}_x(I) \geq m/4.
	\]
	Therefore, by the same computation as above, we have
	\[
	\mathcal{F}_{(\triangle,f)}(u^{(k)}) \geq \dfrac{m\delta}{4} Vol(B).
	\]
\end{proof}

\begin{proof}[\textbf{Proof of Thm \ref{uniformly}}]
	If $(\triangle, f)$ is not uniformly K-stable, then there is a sequence of convex function $u^{(k)}$ with the property
	\[
	\int_{\partial\triangle} \dfrac{u^{(k)}}{f^{2m-1}} d\sigma=1/2, 
	\quad \text { and } \quad 
	\lim _{k \rightarrow \infty} \mathcal{F}_{(\triangle,f)}(u^{(k)})=0.
	\]
	In particular
	\begin{equation}
	\label{limit of integ}
	\lim _{k \rightarrow \infty} \int_{\triangle} \dfrac{s_{(\triangle, \textbf{L}, f)}  u^{(k)}}{f^{2m+1}} d \mu=1.
	\end{equation}
	Since $f$ is positive affine linear function on $\triangle$, then the integration 
	$\int_{\partial\triangle} u^{(k)}d\sigma$
	are uniformly bounded. 
	So $u^{(k)}$ locally uniformly converges to a function $u\in \mathcal{C}^K_{*}$ for some $K>0$.
	
	By Lemma \ref{lower bound}, we know that $\lim_{k\rightarrow \infty}\mathcal{F}_{(c,f)}(u^{(k)})=0$ 
	only if for any interior interval 
	$I\Subset \triangle$ the Monge–Ampere measure of $u|_I$ is 0. 
	If this is the case, $u$ must be affine linear. 
	In fact, suppose we normalize $u$ at some point $p$ such that $u\geq u(p)=0$. 
	Now consider any line $l$ through $p$. By the assumption, the Monge–Ampere measure of $u|_l$ is then trivial. 
	Hence $u\equiv0 $ on $l$. This is true for any line, hence $u\equiv0$. 
	In particular
	\[
	\int_{\triangle} \dfrac{s_{(\triangle, \textbf{L}, f)}  u}{f^{2m+1}} d \mu=0 .
	\]
	Since $u^{(k)}$ locally uniformly converges to $u$, we know
	\[
	\lim _{k \rightarrow \infty} \int_{\triangle}  \dfrac{s_{(\triangle, \textbf{L}, f)}  u^{(k)}}{f^{2m+1}} d \mu
	=\int_{\triangle}\dfrac{s_{(\triangle, \textbf{L}, f)} u}{f^{2m+1}}  d \mu. 
	\]
	Then $\lim _{k \rightarrow \infty} \int_{\triangle}\frac{s_{(\triangle, \textbf{L}, f)} u^{(k)}}{f^{2m+1}} d\mu=0$.
	It contradicts to (\ref{limit of integ}).
\end{proof}

\begin{remark}
	The defnition, Lemma and main Theorem \ref{uniformly} as above are also true if we replace $s_{(\triangle, \textbf{L}, f)}$ by a smooth function $A$ on $\triangle$ which satisfies the moment condition: for each affine linear function $\varphi$,
	\[
	2\int_{\partial\triangle}\varphi \dfrac{d\sigma}{f^{2m-1}}
	=\int_{\triangle}A \varphi \dfrac{d\mu}{f^{2m+1}}.
	\]
\end{remark}

Bohui Chen, An-min Li and Li Sheng \cite{Ch-Li-Sh2} show that uniformly K-stability implies existence of solution of Abreu equation on toric K\"{a}hler surface.
Legendre \cite{LE} shows this result for higher dimension.
By these results, we can set weighted conjecture as follows,
\begin{conjecture}
	If ($\triangle, A, f$) is uniformly K-stable for some smooth function $A$ satisfying the moment condition and positive affine function $f$, then there exists a solution of weighted Abreu equation (\ref{weight Abreu eq}) in $\mathcal{S}$ for $A$.
\end{conjecture}

\section{Uniform K-stability and properness of weighted relative K-energy}
\label{Uniform K-stability and properness of weighted relative K-energy}
We consider the weighted relative K-energy:
\begin{equation}
\label{wet rela K energy}
\mathcal{E}_{(\triangle,f)}(u)
=\mathcal{F}_{(\triangle,f)}(u)
-\int_{\triangle}(\log\det\mathrm{Hess}(u)-\log\det \mathrm{Hess}(u_0)) \dfrac{d\mu}{f^{2m-1}},
\end{equation}
for all $u\in \mathcal{S}$, where $u_0$ is defined by (\ref{canonical potential}). Using the formula 
$d\log\det A=\mathrm{tr} A^{-1}dA$ for any nondegenerate matrix $A$. 
One can compute the first variation of $\mathcal{E}_{(\triangle,f)}$ at $u$ in the direction of $\dot{u}$, 
\begin{align*}
(d\mathcal{E}_{(\triangle,f)})_u(\dot{u})
=&\mathcal{F}_{(\triangle,f)}(\dot{u})
-\int_{\triangle} \dfrac{1}{f^{2m-1}} 
\sum_{i, j=1}^{m} \mathbf{H}_{i j}^{u} 
\dot{u}_{, i j} d \mu            \\
=&\int_{\triangle}
\left[-\sum_{i, j=1}^{m}\left(\dfrac{1}{f^{2 m-1}} \mathbf{H}_{ij}\right)_{,ij}
-\dfrac{s_{(\triangle, \textbf{L}, f)}}{f^{2 m+1}}\right] \dot{u} d \mu,
\end{align*}
showing that the critical points of $\mathcal{E}_{(\triangle,f)}$ are precisely the solutions of (\ref{weight Abreu eq}). Furthermore,
using $dA^{-1}=-A^{-1}dAA^{-1}$, we see that the second variation of $\mathcal{E}_{(\triangle,f)}$ at $u$ in the directions of $\dot{u}$ and $\dot{v}$ is
\[
(d^2\mathcal{E}_{(\triangle,f)})_u(\dot{u},\dot{v})
=\int_{\triangle} \dfrac{1}{f^{2 m-1}} 
\mathrm{tr}\big(\mathrm{Hess}(u)\mathrm{Hess}(\dot{u}) \mathrm{Hess}(u)\mathrm{Hess}(\dot{v})\big) d \mu , 
\]
showing $\mathcal{E}_{(\triangle,f)}$ is convex. 
Thus we can show
\begin{lemma}
	\label{energy lower bd}
	If (\ref{weight Abreu eq}) admits a solution 
	$u\in \mathcal{S}$,
	then the weighted relative K-energy $\mathcal{E}_{(\triangle,f)}$
	attains its minimum at $u$.
\end{lemma} 
\begin{proof}
	We know $\mathcal{E}_{(\triangle,f)}$ is convex on $\mathcal{S}$. The solution $u$ being a critical point of $\mathcal{E}_{(\triangle,f)}$, it is therefore a global minima.
\end{proof}

\begin{proposition}
	\label{unif equi proper}
	For any $\lambda>0$ the following are equivalent:
	\begin{enumerate}[(i)]
		\item \label{uniform K stable}
		$\mathcal{F}_{(\triangle,f)}(u)\geq \lambda\|u\|_b$ for all $u\in \widetilde{\mathcal{C}}$, (i.e. ($\triangle, \mathbf{L}, f$) is uniform K-stable);
		\item \label{coercive} 
		for all $0\leq \delta<\lambda$, there exists $C_{\delta}$ such that 
		$\mathcal{E}_{(\triangle,f)}(u)\geq \delta \|\pi(u)\|_b+C_{\delta}$ for all $u\in \mathcal{S}$, where $\pi(u)$ is the projection of $u$ to the space $\widetilde{\mathcal{S}}:=\widetilde{\mathcal{C}}\cap\mathcal{S}$ (called \textit{properness} of relative K-energy with respect to $\|\cdot\|_b$).
	\end{enumerate}
\end{proposition}
\begin{proof}
	(\ref{uniform K stable})$\Rightarrow$(\ref{coercive}).
	For any smooth bounded function $A$ on $\triangle$, one can define a modified Futaki invariant $\mathcal{F}_{(A,f)}$ by
	\[
	\mathcal{F}_{(A,f)}(u)
	:=2\int_{\partial\triangle}u\dfrac{d\sigma}{f^{2m-1}} 
	-\int_{\triangle}Au\dfrac{d\mu}{f^{2m+1}}.
	\]
	Similarly, one can define a modified weighted relative K-energy $\mathcal{E}_{(A,f)}$ using $\mathcal{F}_{(A,f)}$ instead of $\mathcal{F}_{(\triangle,f)}$ in the formula (\ref{wet rela K energy}).
	
	For any smooth bounded function $A, B$, there is a constant $C=C_{A,B}>0$ with
	\begin{align*}
	\left|\mathcal{F}_{(A,f)}(u)
	-\mathcal{F}_{(B,f)}(u) \right|
	=& \left|-\int_{\triangle}\dfrac{Au}{f^{2m+1}}d\mu 
	+\int_{\triangle}\dfrac{Bu}{f^{2m+1}}d\mu     \right|         \\
	\leq& \int_{\triangle}\dfrac{1}{f^{2m+1}}|A-B| |u| d\mu             \\
	\leq& C_1\|u\|_{L^1} \\
	\leq& C \|u\|_b,
	\end{align*} 
	for all $u\in \widetilde{\mathcal{C}}$, because $\|\cdot\|_b$ bounds the $L^1$ norm on $\widetilde{\mathcal{C}}$.
	Let's take $B=s_{(\triangle, \textbf{L}, f)}$, so that $\mathcal{F}_{(B,f)}=\mathcal{F}_{(\triangle,f)}$, and take $A=s_{u_0,f}$ be the scalar curvature of $(1/f^2)g_{J_0}$, where $J_0$ corresponds to the canonical symplectic potential $u_0\in \mathcal{S}$. Then $u_0$ is trivially solves the equation $s_{u_0,f}=A$.
	
	By assumption,
	\begin{align*}
	\left|\mathcal{F}_{(A,f)}(u)
	-\mathcal{F}_{(\triangle,f)}(u) \right|
	\leq& ((k+1)C-kC) \|u\|_b     \\
	\leq& C(k+1)\lambda^{-1}\lambda\|u\|_b -kC\|u\|_b \\
	\leq& C(k+1)\lambda^{-1}\mathcal{F}_{(\triangle,f)}(u) -kC\|u\|_b,
	\end{align*}
	for all $u\in \widetilde{\mathcal{C}}$.
	This implies
	\[
	\mathcal{F}_{(A,f)}(u)
	\leq (1+C(k+1)\lambda^{-1})\mathcal{F}_{(\triangle,f)}(u) -kC\|u\|_b,
	\] 
	turning this around,
	\[
	\mathcal{F}_{(\triangle,f)}(u)\geq \varepsilon\mathcal{F}_{(A,f)}(u)+\delta\|u\|_b,
	\]
	where $0<\varepsilon:=(1+C\lambda^{-1}(1+k))^{-1}$ for $k$ large enough and 
	$\delta:=kC\lambda(\lambda+C(1+k))^{-1}$, notice  that $\delta$ is an injective function of $k\in [0,\infty)$ with range $[0,\lambda)$.
	
	We know,  
	$\mathcal{F}_{(\triangle,f)}(\varphi)=0$ for all affine function $\varphi$. 
	And from (\ref{Fut if exist solu}), we have, $\mathcal{F}_{(A,f)}(\varphi)=0$ for all affine linear function $\varphi$ since the equation $s_{u_0,f}=A$ admits a trivial solution $u_0$. 
	For any $u\in \mathcal{S}$, by definition, 
	$u=\pi(u)+\varphi$ for some affine function $\varphi$,
	\begin{align*}
	\mathcal{E}_{(\triangle,f)}(u)
	=&\mathcal{F}_{(\triangle,f)}(u)
	-\int_{\triangle}(\log\det\mathrm{Hess}(u)-\log\det \mathrm{Hess}(u_0)) \dfrac{d\mu}{f^{2m-1}}         \\
	=& \mathcal{F}_{(\triangle,f)}(\pi(u))
	-\int_{\triangle}(\log\det \mathrm{Hess}(u)-\log\det \mathrm{Hess}(u_0)) \dfrac{d\mu}{f^{2m-1}}        \\
	\geq& \varepsilon\mathcal{F}_{(A,f)}(u)+ \delta\|\pi(u)\|_b-
	\int_{\triangle}(\log\det\mathrm{Hess}(u)-\log\det \mathrm{Hess}(u_0)) \dfrac{d\mu}{f^{2m-1}}         \\
	=&\mathcal{E}_{(A,f)}(\varepsilon u) -m\log\varepsilon\int_{\triangle}\dfrac{d\mu}{f^{2m-1}} 
	+\delta\|\pi(u)\|_b                 \\
	\geq&\mathcal{E}_{(A,f)}(\varepsilon u) 
	+\delta\|\pi(u)\|_b.
	\end{align*}
	It is shown in Lemma \ref{energy lower bd} that $\mathcal{E}_{(A,f)}$ is bounded from below on the space $\mathcal{S}$. Letting $C_{\delta}$ be a lower bound of $\mathcal{E}_{(A,f)}$, the claim follows.
	
	(\ref{coercive})$\Rightarrow$(\ref{uniform K stable}). 
	Suppose $\mathcal{E}_{(\triangle,f)}(u)\geq \delta \|u\|_b+C_{\delta}$ for all normalized $u\in \mathcal{S}$. 
	We shall fix one such $u$. 
	Then for all $k>0$ and $v\in\widetilde{\mathcal{C}}$, $u+kv\in\widetilde{\mathcal{S}}$
	and so 
	$\mathcal{E}_{(\triangle,f)}(u+kv)\geq \delta \|u+kv\|_b+C_{\delta}$. We have
	\begin{align*}
	k \mathcal{F}_{(\triangle, f)}(v)
	=&\mathcal{E}_{(\triangle, f)}(u+k v) -\mathcal{E}_{(\triangle, f)}(u)
	+\int_{\triangle} \log \left(\dfrac{\det \mathrm{Hess}(u+k v)}{\det \mathrm{Hess}(u)}\right) \dfrac{d\mu}{f^{2m-1}}           \\
	\geq&\delta \|u+kv\|_b+C_{\delta} +\int_{\triangle} \log \left(\dfrac{\det \mathrm{Hess}(u+k v)}{\det \mathrm{Hess}(u)}\right) \dfrac{d\mu}{f^{2m-1}}  -\mathcal{E}_{(\triangle, f)}(u)          \\
	\geq&\delta \|u+kv\|_b+\tilde{C}_{\delta},
	\end{align*}
	where $\tilde{C}_{\delta}=C_{\delta}- \mathcal{E}_{(\triangle, f)}(u)$
	(for the fixed normalized $u\in\mathcal{S}$), since the ratio of the
	determinants is at least one for $k$ sufficiently large. Dividing by $k$ and letting
	$k\rightarrow\infty$ we obtain $\mathcal{F}_{(\triangle,f)}(v) \geq \delta\|v\|_b$.
	Since this is true for all $0\leq \delta<\lambda$, we have that $\mathcal{F}_{(\triangle,f)}(v) \geq \lambda\|v\|_b$ for all $v\in \widetilde{\mathcal{C}}$.
\end{proof}

\begin{proposition}
	If $v\in\mathcal{S}$ is the solution of the weighted Abreu equation (\ref{weight Abreu eq}), then there exists a constant $C$ depending on $\lambda$ such that
	\[
	\|v\|_b\leq C.
	\]
\end{proposition}
\begin{proof}
	Note that by (\ref{Fut if exist solu}), 
	$\mathcal{F}_{(\triangle, f)}(v)
	=n\int_{\triangle}\frac{dx}{f^{2m-1}}=:C_1$.
	Then by the uniform K-stability, we have
	\[
	\|u\|_b\leq \lambda^{-1} \mathcal{F}_{(\triangle, f)}(v)
	=\lambda^{-1}C_1.
	\]
\end{proof}


	
\bigskip
\address{ 
	Academy of Mathematics\\
	and System Science, \\
	Chinese Academy of Sciences \\
	Beijing 100190 \\
	China
}
{jiuyaxiong17@mails.ucas.ac.cn}
\end{document}